\documentclass[letterpaper,10pt,conference]{ieeeconf}
\IEEEoverridecommandlockouts
\makeatletter

\let\proof\@undefined
\let\endproof\@undefined
\makeatother

\usepackage{graphicx,psfrag}
\usepackage[noend]{algorithmic}
\usepackage{xspace}
\usepackage{amssymb,amsmath,amscd,mathrsfs}
\usepackage{amsthm}
\usepackage{comment,color}
\usepackage{url}
\usepackage{multirow}

\usepackage[bookmarks=true]{hyperref}
\usepackage{fixltx2e}

\usepackage{color}
\usepackage{amsmath}
\usepackage{amssymb}
\usepackage{graphicx}
\usepackage{comment,xspace}
\usepackage{fancybox}

\newcommand{\real}{{\mathbb{R}}}
\newcommand{\reals}{\real}

\newtheorem{theorem}{Theorem}[section]
\newtheorem{proposition}[theorem]{Proposition}

\newtheorem{remark}[theorem]{Remark}

\newcommand\oprocendsymbol{\hbox{$\bullet$}}
\newcommand\oprocend{\relax\ifmmode\else\unskip\hfill\fi\oprocendsymbol}

\graphicspath{{./fig/}}

\title{Rebalancing the Rebalancers: Optimally Routing \\Vehicles and
  Drivers in Mobility-on-Demand Systems}

\author{Stephen L. Smith\thanks{S. L. Smith is with the Department of Electrical and Computer Engineering, University of Waterloo, Waterloo ON, N2L 3G1 Canada (\tt stephen.smith@uwaterloo.ca).} \and Marco Pavone\thanks{M. Pavone is with the Department of Aeronautics and Astronautics, Stanford University, Stanford, CA 94305, USA (\tt pavone@stanford.edu).}\and Mac Schwager\thanks{M. Schwager is with the Department of Mechanical Engineering and the Division of Systems Engineering, Boston University, Boston, MA 02215, USA, (\tt schwager@bu.edu).} \and Emilio Frazzoli \thanks{ E. Frazzoli is with the Laboratory for Information and Decision Systems, Aeronautics and Astronautics Department, Massachusetts Institute of Technology, Cambridge, MA 02139, USA
    (\tt frazzoli@mit.edu).} \and Daniela Rus \thanks{D. Rus is with the Computer Science and Artificial Intelligence Laboratory, Electrical Engineering and Computer Science Department, Massachusetts Institute of Technology, Cambridge, MA 02139, USA
    (\tt rus@csail.mit.edu).}
    }

\begin{document}
\maketitle

\begin{abstract}
  In this paper we study rebalancing strategies for a
  mobility-on-demand urban transportation system blending
  customer-driven vehicles with a taxi service. In our system, a
  customer arrives at one of many designated stations and is
  transported to any other designated station, either by driving
  themselves, or by being driven by an employed driver.  The system
  allows for one-way trips, so that customers do not have to return to
  their origin.  When some origins and destinations are more popular
  than others, vehicles will become unbalanced, accumulating at some
  stations and becoming depleted at others. This problem is addressed
  by employing rebalancing drivers to drive vehicles from the popular
  destinations to the unpopular destinations.  However, with this
  approach the rebalancing drivers themselves become unbalanced, and
  we need to ``rebalance the rebalancers'' by letting them travel back
  to the popular destinations with a customer. Accordingly, in this
  paper we study how to optimally route the rebalancing vehicles and
  drivers so that stability (in terms of boundedness of the number of
  waiting customers) is ensured while minimizing the number of
  rebalancing vehicles traveling in the network and the number of
  rebalancing drivers needed; surprisingly, these two objectives are
  aligned, and one can find the optimal rebalancing strategy by
  solving two decoupled linear programs.  Leveraging our analysis, we
  determine the minimum number of drivers and minimum number of
  vehicles needed to ensure stability in the system. Interestingly,
  our simulations suggest that, in Euclidean network topologies, one
  would need between 1/3 and 1/4 as many drivers as vehicles.
\end{abstract}

\section{Introduction}

In this paper we study vehicle routing algorithms for a novel model of
urban transportation system, which involves blending customer-driven
vehicles with a taxi service. Our proposed car-share system is an
example of a Mobility-on-Demand (MOD) system, and aims at providing
urban dwellers with the tailored service of a private automobile,
while utilizing limited urban land more efficiently (e.g., by
minimizing the automobiles that sit unused) \cite{Mitchel.Bird.ea:10}. In
our system, a customer arrives at one of many designated stations and
is transported to any other designated station, either by driving
themselves, or by being driven by an employed driver.  The system
allows for one way trips, so that customers do not have to return to
the same stations from which they picked up their vehicles.  In a
typical one way car-share system (e.g. Car2Go) it has been observed
empirically~\cite{C2G:11}, and shown
analytically~\cite{SLS-MP-EF-DR:11a}, that vehicles become unbalanced,
accumulating at popular destinations and becoming depleted at less
popular ones.  Our proposed system addresses this problem by employing
rebalancing drivers to drive vehicles from the popular destinations to
the unpopular destinations.  However, with this approach the
rebalancing drivers themselves become unbalanced, and hence we need to
``rebalance the rebalancers'' by letting them travel back to the
popular destinations with a customer.  In such a trip, the rebalancing
driver operates the vehicle as a taxi, driving the customer to their
desired destination.  The system is illustrated in
Fig.~\ref{fig:load_balancing}. The main difficulty in such a system,
and the focus of this paper, is how to determine the rebalancing trips
and the taxi trips in order to minimize wasted trips, while providing
the best possible customer experience.

Specifically, the contribution of this paper is twofold: we study
routing algorithms for the MOD system illustrated in
Fig.~\ref{fig:load_balancing} that (1) minimize the number of
rebalancing vehicles traveling in the network, (2) minimize the number
of drivers needed, and (3) ensure that the number of waiting customers
remains bounded. Second, leveraging our analysis, we determine the
relation between the minimum number of drivers needed and the minimum
number of vehicles needed to ensure stability in the system; these
relations would provide a system designer with essential structural
insights to develop business models. Interestingly, our simulations
suggest that, in Euclidean network topologies, one would need between
1/3 and 1/4 as many drivers as vehicles, and that this fraction
decreases to about 1/5 if one allows up to 3-4 drivers to take a trip
with a customer.

This paper builds upon the previous work of the authors in designing
optimal rebalancing policies for MOD systems leveraging
\emph{autonomous operation} of the
vehicles~\cite{SLS-MP-EF-DR:10j,SLS-MP-EF-DR:11a}, i.e., without the
need of human drivers.  On the contrary, the system proposed in this
paper would use technology that is available today (i.e., by employing
human drivers instead of autonomous cars), and our finding are readily
applicable to \emph{existing} one-way car-share systems, which already
employ drivers to rebalance cars using heuristic
methods~\cite{C2G:11}.  Furthermore, by comparing the results in this
paper with those in \cite{SLS-MP-EF-DR:10j}, one can quantitatively
assess the relative benefits of ``hi-tech'' autonomous MOD systems
versus ``low-tech'' driver-based MOD systems. The problem addressed in
this paper has also many characteristics in common with the well-known
Dynamic Traffic Assignment (DTA) problem
\cite{Merchant.Nemhauser:TS:78,Friesz.Luque.ea:OR89,Ziliaskopoulos:TS00,
  Srinivas.Ziliaskopoulos:NSE01}. The key difference between
rebalancing in MOD systems and the DTA problem is that in the former
the optimization is over the empty vehicle trips (i.e., the
rebalancing trips) rather than the passenger carrying trips.

The rest of the paper is structured as follows. In
Section~\ref{sec:model} we present a model for our system with
customers, vehicles, and drivers represented as a continuous fluid,
and we formally state the problem of rebalancing the vehicles and the
drivers. In Section~\ref{sec:properties} we (i) study the
well-posedness of the model and characterize its set of equilibria;
(ii) determine the minimum number of vehicles and drivers needed to
meet the customer demand; and (iii) show that with rebalancing
vehicles and drivers the system is indeed locally stable (i.e., stable
within a neighborhood of the nominal conditions).  In
Section~\ref{sec:opt_reb} we show how to optimally route the
rebalancing vehicles and drivers so that stability (in terms of
boundedness of the number of waiting customers) is ensured while
minimizing the number of rebalancing vehicles traveling in the network
and the number of rebalancing drivers needed; remarkably, these two
objectives are aligned, and one can find the optimal rebalancing
strategy by solving two decoupled linear programs. In
Section~\ref{sec:sim} we study the relation between the minimum number
of drivers needed and the minimum number of vehicles needed. Finally,
in Section~\ref{sec:conc} we give conclusions and discuss future
research directions.

\begin{figure}[htb]
\centering
  \includegraphics[width=0.9\linewidth]{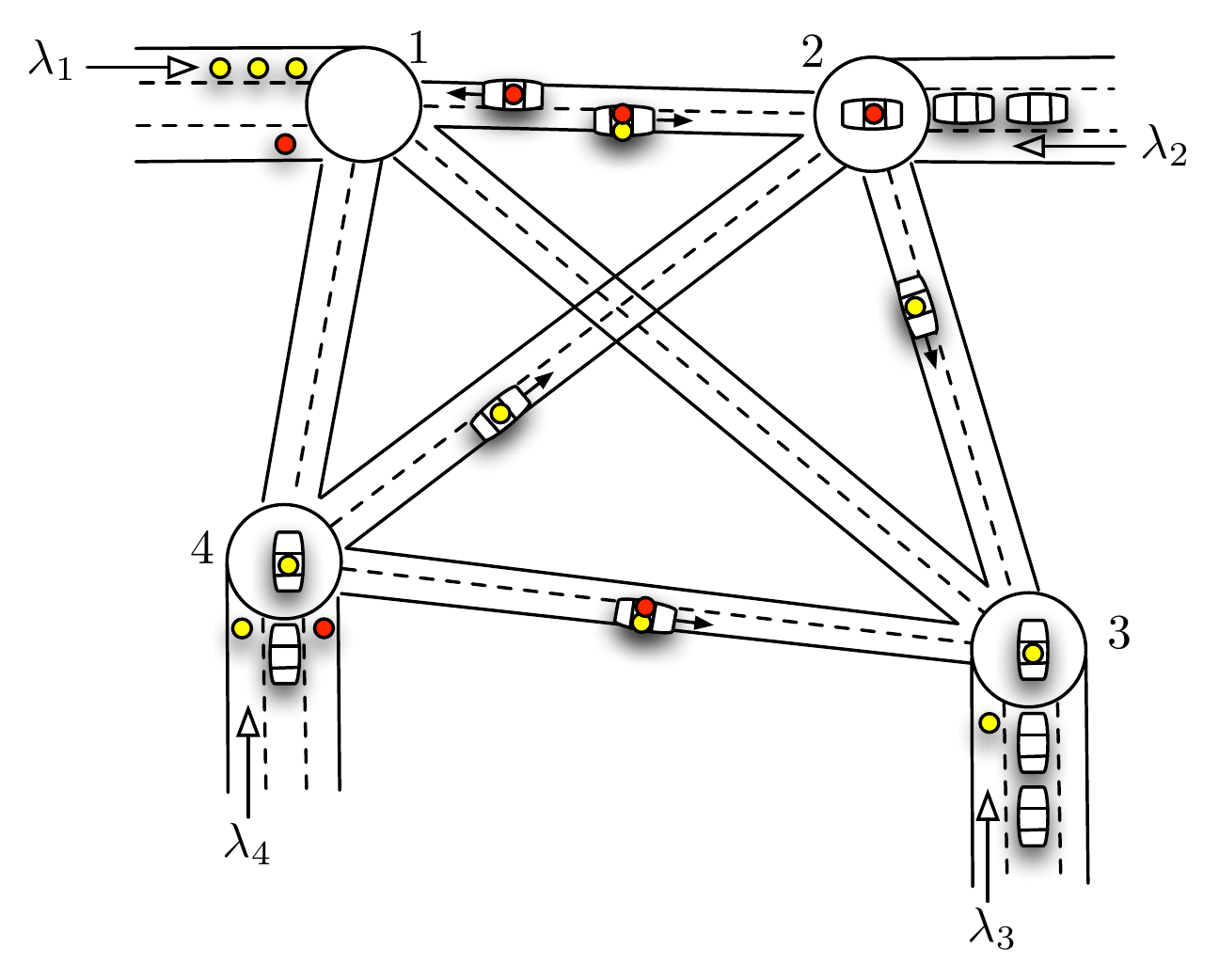}
\caption{At each station there are three queues: customers (yellow
  dots), drivers (red dots), and vehicles (small car icons).  There
  are three modes of use for a car:  A customer can drive a car
  between stations; a customer can be driven between stations by a
  driver; or, a driver can drive a car between stations to rebalance.}
  \label{fig:load_balancing}
\end{figure}

 \begin{table}[htb]
 \small
 \centering %
 \caption{Description of notation for station $i$}%
 \label{tab:parameters}
 \begin{tabular}{c|l}
   {\bf } & {\bf Definition} \\
   \hline
   $c_i$ & number of customers at station $i$\\
   $v_i$ & number of vehicles at station $i$\\
   $r_i$ & number of drivers at station $i$\\
   $\lambda_i$ & rate of arrival of customers at station $i$\\
   $\mu_i$ & departure rate from station $i$ \\
   $T_{ij}$ & travel time from station $i$ to station $j$ \\
   $p_{ij}$ & fraction of customers at station $i$ destined for station $j$ \\
    $\alpha_{ij}$ & rate of rebalancing vehicles from station $i$ to station $j$ \\
   $\gamma_i$ & $\sum_j \alpha_{ij}$ \\
   $\beta_{ij}$ & rate of rebalancing drivers from station $i$ to station $j$ \\
   $f_{ij}$ & fraction of customers traveling from $i$ to $j$ willing\\
   & to use taxis \\
   $H(\cdot)$ & Heaviside function \\
 \end{tabular}
 \end{table}

\section{Modeling the Mobility-on-Demand System}
\label{sec:model}

In our prior work~\cite{SLS-MP-EF-DR:11a} we proposed a fluid model
for mobility-on-demand systems and formulated a policy to optimally
rebalance vehicles assuming that they could operate autonomously.  In
this paper we consider rebalancing the vehicles through the use of
dedicated personnel that are employed to drive the vehicles.  In this
section we extend the fluid model in~\cite{SLS-MP-EF-DR:11a} to
capture the later scenario.

\textbf{Basic model:} The model in \cite{SLS-MP-EF-DR:11a} can be
formalized as follows.  Consider a set of $n$ stations, $\mathcal N =
\{1, \ldots, n\}$, defined over an extended geographical area (see
Figure~\ref{fig:load_balancing}). Since the model is a fluid
approximation, the number of customers, vehicles, and drivers are
represented by real numbers. Customers arrive at station $i$ at a
constant rate $\lambda_i\in \reals_{>0}$.  The number of customers at
station $i$ at time $t$ is $c_i(t) \in \reals_{\geq0}$, and the number
of vehicles waiting idle at station $i$ at time $t$ is $v_i(t)\in
\reals_{\geq 0}$. The total number of vehicles in the system is $V\in
\reals_{>0}$. The fraction of customers at station $i$ whose
destination is station $j$ is $p_{ij}$ (where $p_{ij}\in\reals_{\geq
  0}$, $p_{ii}=0$, and $\sum_{j}p_{ij}=1$). The travel time from
station $i$ to station $j$ is $T_{ij}\in \reals_{\geq 0}$. When there
are both customers and vehicles at station $i$ (i.e., $c_i(t)>0$ and
$v_i(t) >0$), then the rate at which customers (and hence vehicles)
leave station $i$ is $\mu_i$; when, instead, $c_i(t) = 0$ but
$v_i(t)>0$ the departure rate is $\lambda_i$. A necessary condition
for the total number of customers at station $i$ to remain bounded is
that $\mu_i \geq \lambda_i$; we will assume $\mu_i > \lambda_i$
throughout the paper (the case $\mu_i = \lambda_i$ can be addressed
with techniques similar to the ones introduced in this paper and is
omitted).

From~\cite{SLS-MP-EF-DR:11a}, we showed that a station is in need of rebalancing if
$-\lambda_i +\sum_{j\neq i} \lambda_j p_{ji} \neq 0$.  This can be
easily understood by noting that $\lambda_i$ is the rate at which
vehicles leave station $i$, while $\sum_{j\neq i} \lambda_j p_{ji}$ is
the rate at which vehicles arrive at station $i$.  In what follows we
assume that 
\[
-\lambda_i +\sum_{j\neq i} \lambda_j p_{ji} \neq 0 \quad \text{for all
$i\in \mathcal{N}$},
\]
and thus each station is in need of rebalancing.  We comment further
on this assumption in Remark~\ref{rem:balanced_stations}.

\textbf{Rebalancing vehicles:} In order to rebalance the number of
vehicles $v_i(t)$ at each station, vehicles without customers will be
driven between stations using hired human drivers. The number of
drivers waiting at station $i$ is $r_i(t) \in\reals_{\geq 0}$ and the
total number of drivers in the system is $R\in\reals_{>0}$.  In order
to send a vehicle without a customer on a rebalancing trip from
station $i$ to station $j$, there must be a driver present at station
$i$.  We let $\alpha_{ij}\in\reals_{\geq 0}$ denote the rate at which
we send vehicles from station $i$ to station $j$ when vehicles and
drivers are available at station $i$.  The total rate at which station
$i$ sends vehicles without customers is $\gamma_i :=
\sum_{j}\alpha_{ij}$, where $\alpha_{ii} = 0$. We let $\alpha$ denote
the matrix with entries given by $\alpha_{ij}$.  These trips are shown
in Figure~\ref{fig:load_balancing} as vehicles with red dots in them.

\textbf{Rebalancing drivers:} Finally, we must rebalance the drivers
in the network, as they will tend to accumulate at some stations and
become depleted at others.  This is done as follows.  If a driver
would like to make a trip from station $i$ to station $j$, it can
drive a car \emph{for} a customer on a trip from $i$ to $j$, thereby
acting as a taxi driver for that trip.  This allows the driver to make
the journey from station $i$ to station $j$ by ``hitching a ride'' on
a passenger-carrying trip, but without negatively affecting the
customer experience.  We quantify this using two sets of variables.
The variables $\beta_{ij}\in\reals_{\geq 0}$ give the rate at which
drivers are sent from station $i$ to station $j$ when there are idle
drivers available at station $i$.  We let $\beta$ denote the matrix
with entries given by $\beta_{ij}$ and assume $\beta_{ii}=0$.  

The quantities $f_{ij}\in(0,1]$ give the fraction of customers making
the trip from station $i$ to $j$ that would be willing to use the taxi
mode of service on their trip.  The remaining fraction of customers
$1-f_{ij}$ would prefer to drive themselves on their trip. Thus,
$f_{ij}$ imposes a constraint on the largest value of $\beta_{ij}$.
In what follows we assume that the $f_{ij}$ are such that there are
enough customer trips available to rebalance the drivers.  In
Proposition~\ref{prop:feas_exist} we give a necessary and sufficient
condition on the $f_{ij}$ such that this is true.  These trips are
shown in Figure~\ref{fig:load_balancing} as vehicles with red and
yellow dots in them.

The notation is summarized in Table~\ref{tab:parameters}.

We are now ready to write the differential equations governing the
evolution of the number of vehicles, customers, and drivers at each
station.  In order to write the expressions more compactly, we
introduce the following notation:
\begin{align*}
&v_i:=v_i(t), \quad c_i:=c_i(t), \quad r_i:=r_i(t), \\
&v_j^i:=v_j(t-T_{ji}), \quad c_j^i :=c_j(t-T_{ji}), \quad r_j^i:=r_j^i(t-T_{ij}).
\end{align*}
(In other words, $v_j^i$ denotes the number of vehicles that
\emph{were} present at station $j$, specifically $T_{ji}$ time units
\emph{prior} to the current time.) Then, we can write the customer
dynamics at station $i$ as
\[
\dot c_i = 
\begin{cases}
\lambda_i, & \text{if $v_i = 0$}, \\
0, & \text{if $v_i >0$ and $c_i = 0$}, \\
\lambda_i - \mu_i, & \text{if $v_i >0$ and $c_i >0$}.
\end{cases}
\]
Defining the Heaviside function as
\[
H(x) := \left\{ \begin{array}{rl}
 1, &\mbox{ if $x>0$}, \\
  0, &\mbox{ otherwise},
\end{array} \right.
\]
the customer dynamics can be written as
\[
\dot c_i = \lambda_i\big(1 - H(v_i)\big) + (\lambda_i -
\mu_i)H(c_i)H(v_i).
\]

The rate of change of vehicles at station $i$ can be written as the
sum of four components:
\begin{enumerate}
\item the rate at which customer-carrying vehicles depart station $i$:
\[
\begin{cases}
0, & \text{if $v_i =0$} \\
-\lambda_i, & \text{if $v_i > 0$ and $c_i = 0$}, \\
-\mu_i, & \text{if $v_i >0$ and $c_i > 0$},
\end{cases} 
\]
which can be written more compactly as
$
-\lambda_i H(v_i)  + (\lambda_i - \mu_i)H(c_i)H(v_i);
$
\item the rate at which customer-carrying vehicles arrive at station $i$:
\[
\sum_{j\neq i} p_{ji}\,\Bigl(  \lambda_j H(v_j^i) - (\lambda_j - \mu_j)H(c_j^i)H(v_j^i)  \Bigr);
\]
\item the rate at which vehicles without a customer (rebalancing
  vehicles) depart station $i$, given by $-\gamma_i H(v_i)H(r_i)$;
\item the rate at which vehicles without a customer (rebalancing
  vehicles) arrive at station $i$, given by $\sum_{j\neq i}
  \alpha_{ji}H(v_j^i)H(r_j^i)$.
\end{enumerate}
Thus, the vehicle dynamics can be written as
\begin{multline*}
\dot v_i = -\lambda_i H(v_i)  + (\lambda_i - \mu_i)H(c_i)H(v_i) \\ + \sum_{j\neq i} p_{ji}\,\Bigl(  \lambda_j H(v_j^i)  - (\lambda_j - \mu_j)H(c_j^i)H(v_j^i)  \Bigr)  \\-\gamma_i H(v_i) H(r_i) +  \sum_{j\neq i} \alpha_{ji}H(v_j^i) H(r_j^i),
\end{multline*}

Finally, the dynamics for the drivers contains four components.  The
first two components are identical to those of the rebalancing
vehicles, given by 3) and 4) above. (This is due to the fact that each
rebalancing vehicle contains a driver).  The third component is the
rate at which rebalancing drivers depart station $i$ (by driving
customer carrying vehicles): $-\sum_{j\neq i} \beta_{ij} H(v_i)
H(r_i)$.  The fourth term is the rate at which rebalancing drivers
arrive at station $i$ with a customer: $\sum_{j\neq i}
\beta_{ji}H(v_j^i) H(r_j^i)$.   Since drivers rebalance by driving
vehicles on customer trips, we have from the customer dynamics $\dot
c_i$ that
\[
\beta_{ij} \leq
\begin{cases}
f_{ij} \lambda_{i} p_{ij} & \text{if $c_i = 0$} \\
f_{ij} \mu_{i} p_{ij} & \text{if $c_i > 0$} \\
\end{cases}
\]
However, we will consider fixed values of $\beta_{ij}$, and since
$\mu_i > \lambda_i$, we simply need to enforce the more stringent
constraint $\beta_{ij} \leq f_{ij} \lambda_{i} p_{ij}$.

Therefore, the $\dot r_i$ dynamics can be written as
\begin{multline*}
\dot r_i = -\gamma_i H(v_i) H(r_i) +  \sum_{j\neq i} \alpha_{ji}H(v_j^i) H(r_j^i) \\ -\sum_{j\neq i} \beta_{ij} H(v_i) H(r_i) + \sum_{j\neq i} \beta_{ji}H(v_j^i) H(r_j^i).
\end{multline*} 

Putting everything together, we can write a set of nonlinear,
time-delay differential equations describing the evolution of
customers and vehicles in the system
as 
\begin{equation}
\label{eq:model}
\begin{split}
\dot c_i =& \lambda_i\big(1 - H(v_i)\big)  + (\lambda_i - \mu_i)H(c_i)H(v_i),\\
\dot v_i =& -\lambda_i H(v_i)  + (\lambda_i - \mu_i)H(c_i)H(v_i) + \\ &\sum_{j\neq i} p_{ji}\,\Bigl(  \lambda_j H(v_j^i)   - (\lambda_j - \mu_j)H(c_j^i)H(v_j^i)  \Bigr) \\&\qquad  -\gamma_i H(v_i) H(r_i) +  \sum_{j\neq i} \alpha_{ji}H(v_j^i)  H(r_j^i),\\
\dot r_i =& -\gamma_i H(v_i) H(r_i) +  \sum_{j\neq i} \alpha_{ji}H(v_j^i) H(r_j^i) \\ &-\sum_{j\neq i} \beta_{ij} H(v_i) H(r_i) + \sum_{j\neq i} \beta_{ji}H(v_j^i) H(r_j^i).
\end{split}
\end{equation}
where $t \geq 0$; the initial conditions satisfy $c_i(\tau) =0, \, v_i(\tau) = 0$, $r_i(\tau)=0$ for $\tau \in [- \max_{i,j}\, T_{ij}, \, 0)$, $c_i(0) \in \reals_{\geq 0}, \, v_i(0) \in \reals_{\geq 0}$ with $v_i(0)>0$ for at least one $i \in \mathcal N$,  $r_i(0) \in \reals_{\geq 0}$ with $r_i(0)>0$ for at least one $i \in \mathcal N$, and $\sum_{i} \, v_i(0) = V$ and $\sum_i r_i(0) = R$.  The optimization variables $\alpha$ and $\beta$ are constrained as follows:
\begin{align*}
0 \leq &\,\beta_{ij} \leq f_{ij} \lambda_ip_{ij} \\
0\leq &\,\alpha_{ij}.
\end{align*}

The problem we wish to solve is as follows: find an \emph{optimal} vehicle rebalancing assignment $\alpha$ and driver rebalancing assignment $\beta$ that simultaneously
\begin{enumerate}
\item minimizes the number of rebalancing vehicles traveling in the network, 
\item minimizes the number of drivers needed, \emph{and}
\item ensures that the number of waiting customers remains bounded.
\end{enumerate}
Note that this is a multi-objective optimization, and thus it is not clear that one can both minimize the number of rebalancing vehicles in the network and the number of drivers needed.  However, it will turn out that these two objectives are aligned, and one can find an assignment $(\alpha,\beta)$ that minimizes both objectives.

\section{Well-posedness, Equilibria, and Stability of Fluid Model}
\label{sec:properties}
In this section we first discuss the well-posedness of model
\eqref{eq:model} by showing two important properties, namely existence
of solutions and invariance of the number of vehicles and rebalancing
drivers along system trajectories. Then, we characterize the
equilibria, we determine the minimum number of
vehicles and drivers  to ensure their existence, and we give a necessary and sufficient condition on the ``user's preference" $f_{ij}$ such that there are enough customer trips available to rebalance the drivers. Finally, we show that rebalancing vehicles and drivers give rise to
equilibria that are locally (i.e., within a neighborhood of the
nominal conditions) stable.

\subsection{Well-posedness}

The fluid model~\eqref{eq:model} is nonlinear, time-delayed, and the
right-hand side is discontinuous. Due to the discontinuity, we need to
analyze the model within the framework of Filippov solutions (see,
e.g., \cite{Filippov:88}).  The following proposition verifies that the
fluid model is well-posed.
\begin{proposition}[Well-posedness of fluid model]
  \label{thrm:inv}
  For the fluid model~\eqref{eq:model}, the following hold:
  \begin{enumerate}
  \item For every initial condition, there exist continuous functions
    $c_i(t): \reals \to \reals_{\geq 0}$, $v_i(t): \reals \to
    \reals_{\geq 0}$, and $r_i(t): \reals \to
    \reals_{\geq 0}$ $i\in \mathcal N$, satisfying the differential
    equations~\eqref{eq:model} in the Filippov sense. %
  \item The total number of vehicles and rebalancing drivers is
    invariant for $t\geq 0$ and is equal, respectively, to $V = \sum_i
    \, v_i(0)$ and $R = \sum_{i} \, r_i(0)$.
  \end{enumerate}
\end{proposition}
\begin{proof}
  To prove the first claim, it can be checked that all assumptions of
  Theorem II-1 in~\cite{Haddad:IJM81} for the existence of Filippov
  solutions to time-delay differential equations with discontinuous
  right-hand side are satisfied, and the claim follows.

  As for the second claim, the proof of the invariance of the number
  of vehicles is virtually identical to the one of Proposition 3.1 in
  \cite{SLS-MP-EF-DR:11a} and is omitted in the interest of
  brevity. We prove next the invariance of the number of rebalancing
  drivers. Let $r_{ij}(t)$, where $t\geq 0$, be the number of
  rebalancing drivers in-transit from station $i$ to station $j$
  (i.e., the rebalancing drivers for which the last station visited is
  $i$ and the next station they will visit is $j$). Clearly,
  $r_{ii}(t)=0$. Now, the total number $R(t)$ of rebalancing drivers
  in the system at time $t \geq 0$ is given, by definition, by $R(t) =
  \sum_{i=1}^n\, r_i(t) + \sum_{i,j} \, r_{ij}(t)$.  The number of
  in-transit rebalancing drivers at time $t$ is given by the integral
  over the last $T_{ij}$ time units (i.e., the time to get from
  station $i$ to station $j$) of the rebalancing driver departure
  \emph{rate} from station $i$ to station $j$. Such departure rate is
  the sum of the departure rate of rebalancing vehicles (since each
  rebalancing vehicle contains a rebalancing driver) and of the
  departure rate of rebalancing drivers that drive customer-carrying
  vehicles; hence, one can express $r_{ij}(t)$ as
\begin{multline}
\label{eq:transVeh}
r_{ij}(t) = \int_{t-T_{ij}}^t \underbrace{\alpha_{ij}H(v_i(\tau))H(r_i(\tau))}_{\text{rate of drivers on rebalancing vehicles}} \,+ \\
\underbrace{\beta_{ij}H(v_i(\tau))H(r_i(\tau))}_{\text{rate of drivers on customer-carrying vehicles}} \,d\tau.
\end{multline}
By applying the Leibniz integral rule, one can write
\begin{equation*}
\begin{split}
\dot r_{ij}(t) &=  (\alpha_{ij} + \beta_{ij})\Bigl(H(v_i)H(r_i)  - H(v_i^j)H(r_i^j)\Bigr).
 \end{split}
\end{equation*}
Therefore, one immediately obtains, for $t\geq 0$,
\begin{align*}
\dot  R(t) &= \sum_{i=1}^n \dot r_i(t) +\sum_{i=1}^n\sum_{j=1}^n\dot
r_{ij}(t) \\
&= - \sum_{i=1}^n \sum_{j=1}^n (\alpha_{ij} + \beta_{ij})H(v_i)H(r_i)
\, + \\ 
&\qquad \sum_{i=1}^n \sum_{j=1}^n (\alpha_{ji} + 
\beta_{ji})H(v_j^i)H(r_j^i) \,+ \sum_{i=1}^n\sum_{j=1}^n\dot r_{ij}(t) \\&= 0.
\end{align*}

This proves the claim.
\end{proof}

\subsection{Equilibria}

The following result characterizes the equilibria of
model~\eqref{eq:model}.  Recall that no station is exactly balanced,
and thus $-\lambda_i +\sum_{j\neq i} \lambda_j p_{ji} \neq 0$, for all
$i\in \mathcal{N}$.

\begin{theorem}[Existence of equilibria]
  \label{cor:gen}
  Let $\mathcal A \times \mathcal B$ be the set of assignments
  $(\alpha,\beta)$ that verify the equations
\begin{align}
\label{eq:cont_constr}
\sum_{j\neq i}(\alpha_{ij} - \alpha_{ji}) &= D_i,\\
\label{eq:cont_constr2}
\sum_{j\neq i}(\beta_{ij}-\beta_{ji}) &= -D_i,
\end{align}
for each $i \in \mathcal N$, where $D_i :=-\lambda_i +\sum_{j\neq i}
\lambda_j p_{ji}$.  Moreover, let
\begin{align*}
V_{\alpha} &:= \sum_{i,j} \,T_{ij}\,(p_{ij}\lambda_i + \alpha_{ij}),
\quad \text{and} \\
R_{\alpha,\beta} &:= \sum_{i,j} \,T_{ij}\,(\alpha_{ij} + \beta_{ij}).  
\end{align*}
If $(\alpha,\beta) \notin \mathcal A\times \mathcal B$, then no equilibrium exists.  If $(\alpha,\beta) \in \mathcal A\times \mathcal B$, there are two cases:
 \begin{enumerate}
 \item If $V>V_{\alpha}$ and $R > R_{\alpha,\beta}$, then the set of 
  equilibria is\[
  c_i=0, \qquad v_i > 0, \quad r_i >0 \qquad \forall\; i\in \mathcal N,
  \]
  where $\sum_i v_i =  V - V_{\alpha}$ and $\sum_i r_i =  R - R_{\alpha,\beta}$. 
  \item If $V\leq V_{\alpha}$ or $R \leq R_{\alpha,\beta}$, then no equilibrium exists.
  \end{enumerate}
\end{theorem}
\begin{proof}

  To prove the theorem, we set $\dot c_i = 0$, $\dot v_i = 0$, and
  $\dot r_i =0$ for all $i\in \mathcal N$.  From the $\dot c_i = 0$
  equations we obtain
\begin{equation}
\label{eq:lambda_eqm}
\lambda_i = \lambda_i H(v_i) - (\lambda_i - \mu_i) H(v_i)
H(c_i).
\end{equation}
Since $\lambda_i < \mu_i$, the above equations have a solution only if
\[
c_i=0 \quad \text{and} \quad v_i > 0 \quad \text{$\forall\; i \in \mathcal N$}.
\]
Setting $\dot v_i = 0$, combined with \eqref{eq:lambda_eqm} and the
fact that in equilibrium $c_i=0$ and $v_i$ is a positive constant, we
obtain
\begin{equation}
\label{eq:alpha_constraint}
\sum_{j\neq i}\big( \alpha_{ij} H(r_i) - \alpha_{ji} H(r_j) \big) =
D_i,
\end{equation}
where $D_i :=-\lambda_i +\sum_{j\neq i} \lambda_j p_{ji}$.  Finally,
setting $\dot r_i = 0$, combined with the fact that $v_i >0$ in
equilibrium, we obtain
\begin{equation}
\label{eq:beta_constraint}
\begin{aligned}
&\sum_{j\neq i}\big( \alpha_{ij} H(r_i) - \alpha_{ji} H(r_j) \big)\\ &= -
\sum_{j\neq i}\big( \alpha_{ij} H(r_i) - \alpha_{ji} H(r_j) \big) 
= -D_i.
\end{aligned}
\end{equation}

Now, consider any station $i$, and note that by assumption we have
$D_i \neq 0$.  If $D_i >0$ then from~\eqref{eq:alpha_constraint} we
see that $r_i >0$ in equilibrium.  Alternatively, if $D_i < 0$, then
from~\eqref{eq:beta_constraint} we see that $r_i>0$.  Therefore, in
equilibrium $r_i >0$.

We have shown that all equilibria are of the form $c_i=0$, $v_i >0$,
and $r_i>0$, for each $i\in\mathcal{N}$.  A necessary condition for
the existence of equilibria is that the rebalancing assignments
$\alpha$ and $\beta$ can be chosen such that they lie in the set
$\mathcal{A} \times \mathcal{B}$ of assignments that verify
\begin{align*}
\sum_{j\neq i}(\alpha_{ij} - \alpha_{ji}) &= D_i, \\
\sum_{j\neq i}(\beta_{ij}-\beta_{ji}) &= -D_i,
\end{align*}
for each $i\in\mathcal N$.  If $(\alpha,\beta)\notin \mathcal A\times
\mathcal B$, then no equilibrium exists and the first claim is proven.

Assume now that $(\alpha,\beta) \in \mathcal A\times \mathcal B$ and
assume that $V>V_{\alpha}$ and $R> R_{\alpha,\beta}$.  We need to show
that $c_i=0$, $v_i> 0$, and $r_i >0$ for all $i\in \mathcal N$ are
indeed valid equilibria.  The necessary conditions in
equations~\eqref{eq:cont_constr} and~\eqref{eq:cont_constr2} are
clearly satisfied and thus we simply need to verify that the number of
vehicles and drivers are sufficient to support the equilibrium
configuration.  But, we showed in~\cite{SLS-MP-EF-DR:11a} that
$V_{\alpha}$ is exactly the equilibrium number of vehicles in transit.
Similarly, from equation \eqref{eq:transVeh} we can verify that
$R_{\alpha,\beta}$ is the equilibrium number of drivers in transit.
This, together with the invariance result in Theorem \ref{thrm:inv},
shows the second claim.
  
  Finally, we can show that if $(\alpha,\beta) \in
  \mathcal A \times \mathcal B$ but $V\leq V_{\alpha}$ or $R\leq
  R_{\alpha,\beta}$, then no equilibrium exists, by arguing that in
  this case there is not a sufficient number of vehicles and/or
  drivers to support the equilibrium.
\end{proof}

\begin{remark}[Balanced stations case]
\label{rem:balanced_stations}
We have assumed that $D_i =-\lambda_i +\sum_{j\neq i}\lambda_jp_{ji} \neq 0$ for
each station $i$.  This assumption removes the pathological case that
a station is perfectly balanced and does not need any rebalancing
effort.  In the case that $D_i=0$ for a station, then $r_i=0$
becomes a valid equilibrium.  Due to space constraints we have omitted
a full treatment of the $D_i=0$ case in this presentation. \oprocend
\end{remark}

One question remains; does there always exist an assignment
$(\alpha,\beta)\in \mathcal A\times \mathcal B$ that satisfies the
constraints $\alpha_{ij} \geq 0$, and $0\leq \beta_{ij} \leq
f_{ij}\lambda_i p_{ij}$ for each $i,j\in \mathcal N$?  We call such an
assignment \emph{feasible}.   It is straightforward to
verify that a feasible assignment for $\alpha$ always exists, since the
variables are constrained only to be non-negative~\cite{SLS-MP-EF-DR:11a}.  The
$\beta$ variables, however, are bounded from above (that is, they have
finite capacities), and thus it is not clear whether there exists a
feasible $\beta$ assignment.  The following result gives a standard
condition for the existence of a feasible assignment (see, for
example~\cite[p.\ 220]{BK-JV:07} and a consequence of this condition.
\begin{proposition}[Existence of a feasible assignment]
  \label{prop:feas_exist}
  A feasible assignment $(\alpha,\beta)$ exists if and only if, 
  \begin{equation}
    \label{eq:feas_existence}
  -\sum_{i\in S} D_i \leq \sum_{i\in S, j\notin S} f_{ij} \lambda_i
  p_{ij} \quad \text{for every set $S\subseteq \mathcal N$},
  \end{equation}
  where $D_i = -\lambda_i + \sum_{j\neq i} \lambda_j p_{ji}$.  As a
  consequence, if $f_{ij} = 1$ for all $i,j,\in\mathcal N$, then a
  feasible assignment always exists.
\end{proposition}
\begin{proof}
  The condition~\eqref{eq:feas_existence} is a standard condition for
  the existence of a feasible solution in a minimum cost flow
  problem~\cite[p.\ 220]{BK-JV:07}.

  Now we show that if $f_{ij}=1$ for all $i,j\in\mathcal N$,
  then~\eqref{eq:feas_existence} is satisfied.  Take any subset
  $S\subseteq \mathcal{N}$ and let us show that
  \begin{align*}
    \sum_{i\in S} D_i + \sum_{i\in S, j\notin S} \lambda_i p_{ij} \geq
    0 .
  \end{align*}
  From the definition of $D_i$, the left-hand side of the above
  expression can be written as
  \begin{align*}
    &- \sum_{i\in S} \lambda_i + \sum_{i\in S,j\in\mathcal N} \lambda_j
    p_{ji} + \sum_{i\in S, j\notin S} \lambda_i p_{ij} \\
    & = - \sum_{i,j\in S} \lambda_ip_{ij} + \sum_{i\in S,j\in\mathcal N} \lambda_j
    p_{ji} \\
    & = - \sum_{i,j\in S} \lambda_ip_{ij} + \sum_{i\in S,j\in\mathcal N} \lambda_j
    p_{ji} \\
    & = \sum_{i\notin S,j\in S} \lambda_i p_{ij} \geq 0.
  \end{align*}
  This proves the feasibility when $f_{ij} = 0$ for all $i,j\in
  \mathcal{N}$. 
\end{proof}

\smallskip

\subsection{Stability of Equilibria}

In this section we investigate the (local) \emph{stability} of the
equilibria of our model. We consider the following notion of local
stability. Let $(\alpha, \beta) \in \mathcal A \times \mathcal B$ and
assume $V>V_{\alpha}$ and $R>R_{\alpha,\beta}$ (this is a necessary
and sufficient condition to have equilibria, see
Theorem~\ref{cor:gen}). We say that the (non-empty) set of equilibria
\begin{equation}
\label{eq:equil_set}
\begin{split}
&\mathcal E_{\alpha, \beta}
:= \big\{(\mathbf{c}, \mathbf{v}, \mathbf{r})\in \reals^{3n}\, \big| \, c_i = 0, v_i>0, r_i>0
\text{ for all } \\&\qquad   i\in \mathcal N, \text{ and }\sum_i v_i =
V-V_{\alpha} \text{ and } \sum_i r_i = R-R_{\alpha,\beta}\big\}
\end{split}
\end{equation}
is locally asymptotically stable if for any equilibrium
$(\underline{\mathbf{c}}, \underline{\mathbf{v}}, \underline{\mathbf{r}}) \in \mathcal
E_{\alpha,\beta} $ there exists a neighborhood $ \mathcal
B^{\delta}_{\alpha, \beta} (\underline{\mathbf{c}},
\underline{\mathbf{v}},\underline{\mathbf{r}}):=\{(\mathbf{c}, \mathbf{v}, \mathbf{r})\in \reals^{3n}\, |
\, c_i \geq 0, v_i\geq 0, r_i\geq 0 \text{ for all } i \in \mathcal N,
\|(\mathbf{c} - \underline{\mathbf{c}}, \mathbf{\mathbf{v}} -
\underline{\mathbf{v}}, \mathbf{\mathbf{r}} -
\underline{\mathbf{r}}) \|<\delta, \text{ and } \sum v_i
= V - V_{\alpha} \text{ and } \sum r_i
= R- R_{\alpha, \beta}\} $ such that every evolution of model \eqref{eq:model}
starting at
\begin{equation}\label{eq:init_stab_0}
\begin{split}
  &c_i(\tau) =  \underline{c}_i \text{ for } \tau \in [-\max_{i,j} T_{ij},\,0)\\
  &v_i(\tau) = \underline{v}_i  \text{ for } \tau \in [-\max_{i,j} T_{ij},\,0)\\
  &r_i(\tau) = \underline{r}_i  \text{ for } \tau \in [-\max_{i,j} T_{ij},\,0)\\
  &(\mathbf{c}(0), \mathbf{v}(0), \mathbf{r}(0)) \in \mathcal
  B^{\delta}_{\alpha, \beta}(\underline{\mathbf{c}}, \underline{\mathbf{v}}, \underline{\mathbf{r}}))
\end{split}
\end{equation}
has a limit which belongs to the equilibrium set. In other words,
$\big(\lim_{t\to+\infty} \mathbf{c}(t),
\lim_{t\to+\infty}\mathbf{v}(t), \lim_{t\to+\infty}\mathbf{r}(t)\big)\in \mathcal E_{\alpha, \beta}$. The next theorem characterizes stability. 
\begin{theorem}[Stability of equilibria]\label{thrm:loc_stability}
 Let $(\alpha,\beta)\in\mathcal A \times \mathcal B$ be a feasible assignment, and assume $V>V_{\alpha}$ and $R>R_{\alpha,\beta}$; then, the set of equilibria $\mathcal E_{\alpha,\beta}$ is locally asymptotically stable.
\end{theorem}
\begin{proof}
  Consider an equilibrium
  $(\underline{\mathbf{c}},\underline{\mathbf{v}}, \underline{\mathbf{r}}) \in \mathcal
  E_{\alpha,\beta}$ (note that $\underline{\mathbf{c}} = \mathbf{0}$ by Theorem \ref{cor:gen}). We
  now prove that every evolution of model \eqref{eq:model} starting at
\begin{equation}
\label{eq:init_stab}
\begin{split}
  &c_i(\tau) = 0 \text{ for } \tau \in [-\max_{i,j} T_{ij},\,0)\\
  &v_i(\tau) = \underline{v}_i  \text{ for } \tau \in [-\max_{i,j} T_{ij},\,0)\\
  &r_i(\tau) = \underline{r}_i  \text{ for } \tau \in [-\max_{i,j} T_{ij},\,0)\\
  &(\mathbf{c}(0), \mathbf{v}(0), \mathbf{r}(0)) \text{ such that } (1) \, \, 0\leq c_i(0) < v_i(0) \, \,  \forall i, \\
  &\qquad \qquad (2) \, \,0<r_i(0) \,\, \forall i, (3)\, \, \sum v_i(0) =  V - V_{\alpha},\\
  & \qquad \qquad \text{and } (4) \, \,  \sum r_i(0) =  R - R_{\alpha,\beta}
\end{split}
\end{equation}
has a limit which belongs to the equilibrium set. The claim of the theorem will
then be an easy consequence of this statement.

We start by observing the following fact. Assume that $v_i(\tau)> 0$ and $r_i(\tau)>0$
for all $\tau\in [-\max_{i,j}T_{ij}, \, t]$, then at time $t$ the
differential equations read
$\dot{c}_i(t) = (\lambda_i - \mu_i)H(c_i(t))$, for all  $i\in \mathcal N$;
recalling that, by Theorem \ref{cor:gen}, it must hold $-\lambda_i+\sum_{j\neq i}\lambda_j p_{ji} - \gamma_i
+ \sum_{j\neq i}\alpha_{ji} = 0$, one can write
\begin{equation*}
\begin{split}
\dot{v}_i(t) &=  -\lambda_i + (\lambda_i - \mu_i)H(c_i)  + \sum_{j\neq i}p_{ji}\Bigl(\lambda_j - \\
&\qquad  \qquad  (\lambda_j - \mu_j)H(c_j^i) \Bigr) - \gamma_i + \sum_{j\neq i} \alpha_{ji}\\
&= (\lambda_i - \mu_i)H(c_i) -  \sum_{j\neq i}p_{ji}  (\lambda_j - \mu_j)H(c_j^i)\\
&\geq  (\lambda_i - \mu_i)H(c_i)  ,\quad  \text{for all } i\in \mathcal N.
\end{split}
\end{equation*}
Also, since by Theorem \ref{cor:gen}, it must hold - $\sum_{j\neq i}(\alpha_{ij} - \alpha_{ji}) + \sum_{j\neq i}(\beta_{ji} - \beta_{ij}) = 0$, one can write
\begin{equation*}
\begin{split}
\dot{r}_i(t) &=  -\sum_{j\neq i} (\alpha_{ij} + \beta_{ij}) + \sum_{j\neq i} (\alpha_{ji} + \beta_{ji})=0.
\end{split}
\end{equation*}

Since $v_i(\tau)>0$ for all $\tau\in [-\max_{i,j}T_{ij}, \, 0]$, and since $v_i(0) >c_i(0)$ for all $i\in \mathcal N$, we conclude that no $v_i(t)$ and $r_i(t)$ can reach the value $0$ \emph{before} the corresponding number of customers $c_i(t)$ has reached the value $0$. However, once $c_i(t)$ reaches the value $0$ (after a time interval $c_i(0)/(\mu_i - \lambda_i)$), the time derivative 
$\dot v_i(t)$ is larger than or equal to zero. This implies that when the initial conditions satisfy \eqref{eq:init_stab}, then $v_i(t)>0$ and $r_i(t)>0$ for all $t\geq0$.

Since $v_i(t)>0$ and $r_i(t)>0$ for all $t\geq0$, and since this implies that
$\dot{c}_i(t) = (\lambda_i - \mu_i)H(c_i(t))$ for all $i\in \mathcal
N$ and $t\geq 0$, we conclude that all $c_i(t)$ will be equal to zero
for all $t\geq T^{\prime}:=\max_{i}\, c_i(0)/(\mu_i-\lambda_i)$. Then,
for $t\geq T^\prime+\max_{ij}T_{ij}=:T^{\prime \prime}$ the
differential equations become:
$\dot{c}_i(t) = 0$,
$\dot{v}_i(t) =  0$, $\dot{r}_i(t) =  0$.

Collecting the results obtained so far, we have that $\lim_{t\to
  +\infty}c_i(t) = 0$ for all $i \in \mathcal N$. Moreover, since
$\dot v_i(t)=0$ and $\dot r_i(t)=0$ for all $t\geq T^{\prime \prime}$, the limits
$\lim_{t\to +\infty } v_i(t)$ and $\lim_{t\to +\infty } r_i(t)$ exist. Finally, one has $ v_i(t) =
v_i(0) + \int_{0}^{t}\, \dot v_i(\tau)\, d\tau\geq v_i(0) +
\int_{0}^{t}\, \dot c_i(\tau)\, d\tau = v_i(0)+c_i(t) - c_i(0)$.
Since $v_i(0) >c_i(0)$, we conclude that $\lim_{t\to+\infty}v_i(t)>0$. Also, $\dot r_i(t)=0$ for all $t$, hence $\lim_{t\to+\infty}r_i(t)>0$.
Thus any solution with initial conditions \eqref{eq:init_stab} has a
limit which belongs to $\mathcal E_{\alpha, \beta}$ (the properties $\lim_{t
  \to +\infty} \sum v_i(t) = V-V_{\alpha}$ and $\lim_{t
  \to +\infty} \sum r_i(t) = R-R_{\alpha,\beta}$ are guaranteed by the
invariance property in Proposition~\ref{thrm:inv} and the assumptions
$\sum v_i(0) = V-V_{\alpha}$ and $\sum r_i(0) = R-R_{\alpha, \beta}$).

Let $\psi_i:=\min(\underline{r_i}, \underline{v_i}\, \sin\frac{\pi}{4})$,
and let $\psi_{\text{min}}:=
\min_i \, \psi_i$. Then, from simple a geometric argument and from the definitions of $\psi_i$ and
$\psi_{\text{min}}$, it follows that if one chooses $\delta =
\psi_{\text{min}}$, then any solution of model \eqref{eq:model} with
initial conditions satisfying \eqref{eq:init_stab_0} has a limit which
belongs to the equilibrium set. This concludes the proof.
\end{proof}

\section{Optimal Rebalancing}
\label{sec:opt_reb}

Our objective is to find a rebalancing assignment $(\alpha,\beta)$
that simultaneously minimizes the number of rebalancing vehicles
traveling in the network and the number of rebalancing drivers needed,
while ensuring the existence of (locally) stable equilibria for model
\eqref{eq:model}. From the previous section, we already know that the
set of assignments ensuring the existence of stable equilibria is
$\mathcal A\times \mathcal B$ (provided that the total number of
vehicles $V$ and drivers $R$ is large enough).

The time-average number of rebalancing vehicles traveling in the
network is simply given by $\sum_{i,j}T_{ij} \alpha_{ij}$. Note that
in minimizing this quantity we are also minimizing the lower bound on
the necessary number of vehicles $V_{\alpha}$.  The time-average
number of drivers in the network is given by $\sum_{i,j} T_{ij}
(\alpha_{ij}+\beta_{ij})$.  Note that in minimizing this quantity we
are minimizing the lower bound on the necessary number of drivers
$R_{\alpha,\beta}$.

Combining the two objectives with the existence of stable equilibria
constraints in~\eqref{eq:cont_constr}
and~\eqref{eq:cont_constr2}), we obtain the following optimization:
\begin{align*}
  \text{minimize} \;\;& \sum_{i,j} T_{ij} \alpha_{ij} \;\text{and} \; \sum_{i,j}T_{ij}(\alpha_{ij}+\beta_{ij})& \\
  \text{subject to} \;\;& \sum_{j\neq i} (\alpha_{ij} - \alpha_{ji}) =
  D_i
  &\forall\; i\in\mathcal{N}\\
  & \sum_{j\neq i} (\beta_{ij} - \beta_{ji}) = -D_i
  &\forall\; i\in\mathcal{N} \\
  & 0 \leq \alpha_{ij} & \forall\; i,j\in\mathcal{N},\\
  & 0 \leq \beta_{ij} \leq f_{ij}\lambda_i p_{ij} & \forall\;
  i,j\in\mathcal{N},
\end{align*}
where $D_i = \lambda_i + \sum_{j\neq i} \lambda_j p_{ji}$, and the
optimization variables are $\alpha_{ij}$ and $\beta_{ij}$, where $i,
j\in \mathcal N$.  The constraints ensure that the optimization is
over the set $\mathcal A\times \mathcal B$.  Note, however, that this
optimization can be decoupled into an optimization over $\alpha$ and
an optimization over $\beta$.  Both optimizations are minimum cost
flow problems~\cite{BK-JV:07}.  The $\alpha$ optimization is identical
to that presented in~\cite{SLS-MP-EF-DR:11a}:
\begin{align*}
  \text{minimize} \;\;& \sum_{i,j} T_{ij} \alpha_{ij} & \\
  \text{subject to} \;\;& \sum_{j\neq i} (\alpha_{ij} - \alpha_{ji}) =
  D_i
  &\forall\; i\in\mathcal{N}\\
  & \alpha_{ij} \geq 0 & \forall\; i,j\in\mathcal{N}.
\end{align*}
The $\beta$ optimization then looks as follows:
\begin{align*}
  \text{minimize} \;\;& \sum_{i,j} T_{ij} \beta_{ij} & \\
  \text{subject to} \;\;& \sum_{j\neq i} (\beta_{ij} - \beta_{ji}) =
  -D_i
  &\forall\; i\in\mathcal{N}\\
  & 0 \leq \beta_{ij} \leq f_{ij}\lambda_i p_{ij} & \forall\; i,j\in\mathcal{N}.
\end{align*}
The $\alpha$ optimization is an uncapacitated minimum cost flow
problem and thus is always feasible.  In
Proposition~\ref{prop:feas_exist} we give conditions on the $f_{ij}$
fractions in order for the $\beta$ optimization to be feasible.  

The rebalancing policy is then given by solving the two minimum cost
flow problems to obtain solutions $\alpha_{ij}^*$ and $\beta_{ij}^*$.
We then send empty rebalancing vehicles (along with drivers) from
station $i$ to station $j$ at a rate of $\alpha^*_{ij}$ (when vehicles
and drivers are available at station $i$).  In addition, we send
drivers on customer-carrying vehicles from $i$ to $j$ at a rate of
$\beta^*_{ij}$ (when customers and vehicles are available at station
$i$).

\section{Simulations}
\label{sec:sim}

\begin{figure*}
\centering
\includegraphics[width=0.48\linewidth]{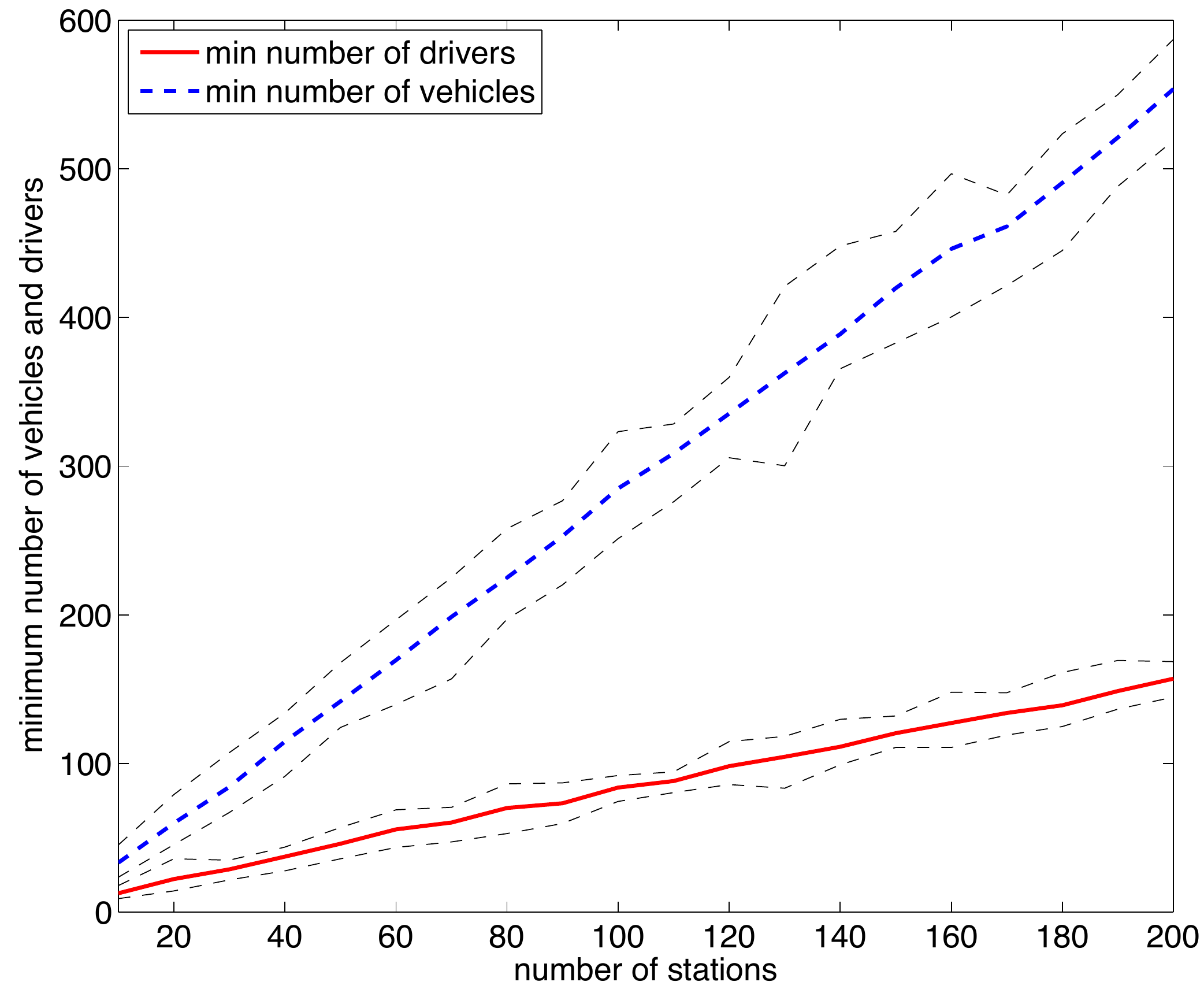} \hfill
\includegraphics[width=0.48\linewidth]{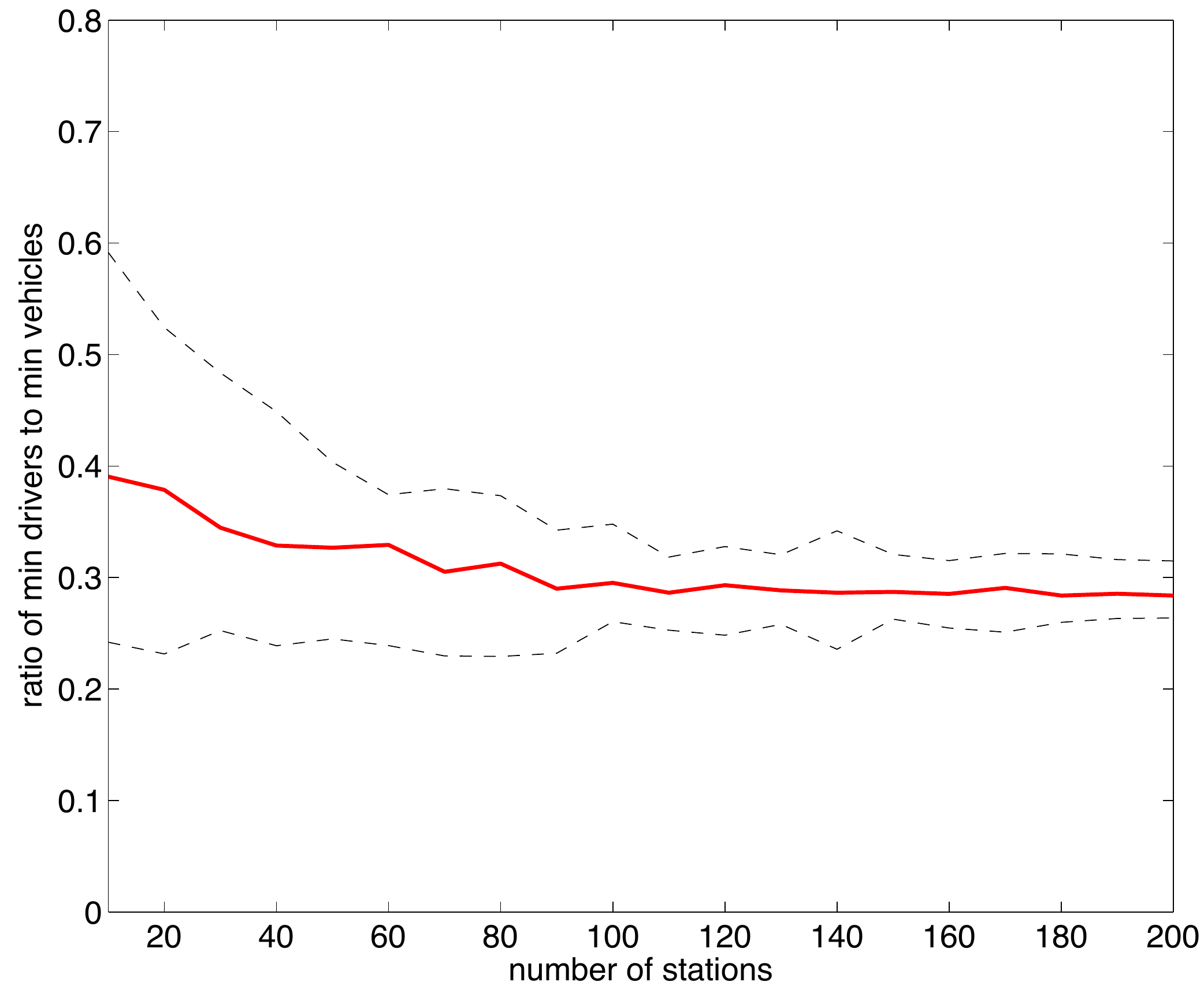} \vskip1em
\includegraphics[width=0.48\linewidth]{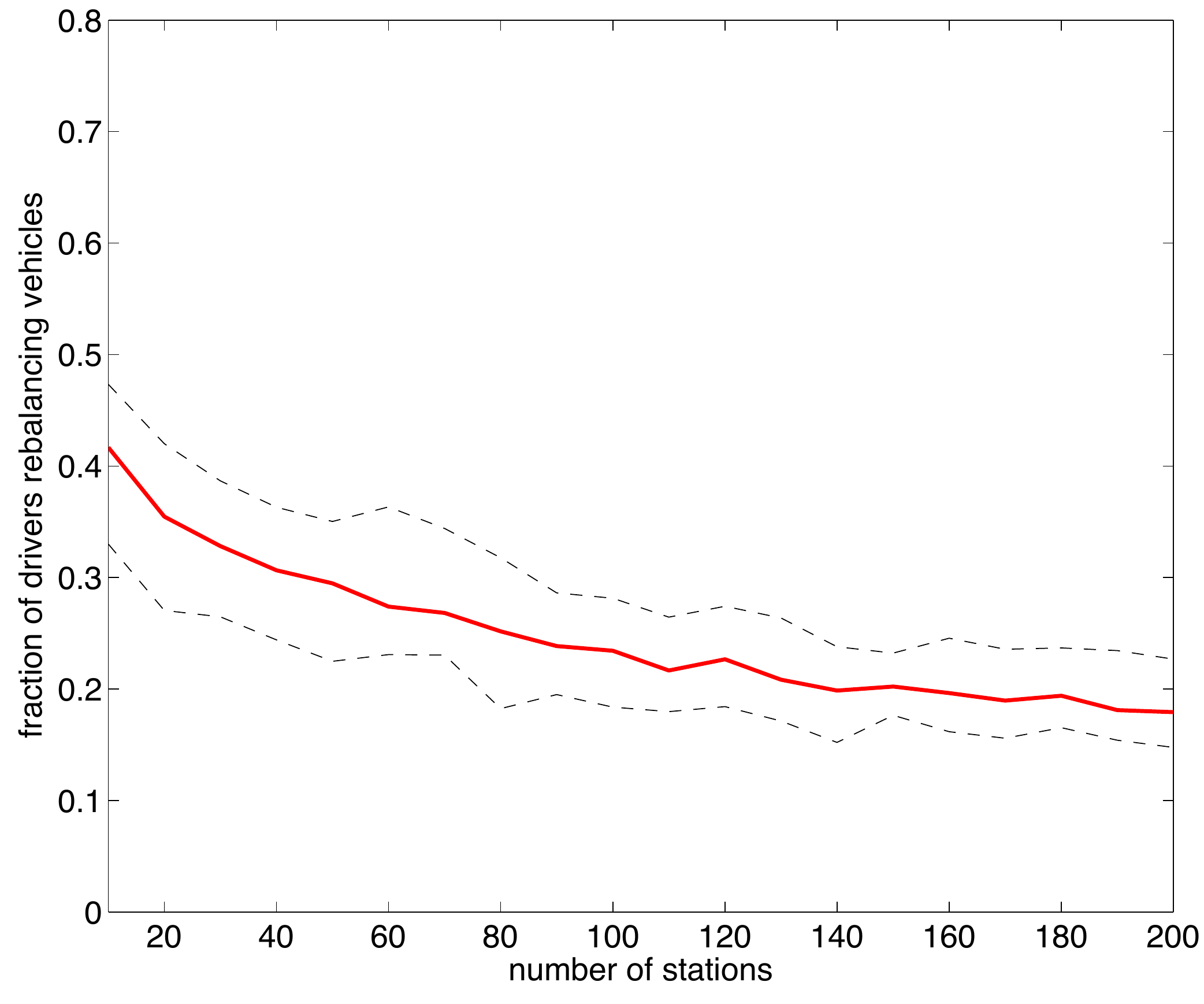}
\caption{Left figure:  The minimum number of vehicles and drivers.
  Middle figure:  The ratio between the minimum number of drivers and
  number of vehicles.  Right figure:  The fraction of drivers that are
performing vehicle rebalancing trips.  For each fixed number of
stations, $20$ trials were performed.  Thick lines show the mean of
the $20$ trials while thin dashed lines show the maximum and minimum
over the trials.}
\label{fig:vehicles_and_drivers}
\end{figure*}

In this section we study the relation between the minimum number of
drivers needed for stability $R_{\alpha,\beta}$ and the minimum number
of vehicles needed $V_{\alpha}$ from Theorem~\ref{cor:gen}. To
evaluate these quantities, we need to generate sample data consisting
of arrival rates $\lambda_i$ at each station $i$, customer destination
probabilities $p_{ij}$, travel times between stations $T_{ij}$, and
the fraction of customers $f_{ij}$ traveling from $i$ to $j$ that are
willing to be driven by a driver.  We generate this data as follows:
We uniformly randomly place $n$ stations in a $100\times 100$
environment, and calculate the travel times $T_{ij}$ as the Euclidean
distance between stations.  We uniformly randomly generate the arrival
rates $\lambda_i$ on the interval $[0,0.05]$ arrivals per time unit.
Similarly we uniformly randomly generate the destination probabilities
$p_{ij}$ such that they are nonnegative and $\sum_j p_{ij} =1$ for
each station $i$.  Finally, we assume that $f_{ij} = 1$ for each pair
of stations in order to avoid issues with feasibility.

To solve the optimizations in Section~\ref{sec:opt_reb} for the
optimal assignment $(\alpha^*,\beta^*)\in\mathcal A \times \mathcal B$, we
use the freely available SeDuMi (Self-Dual-Minimization) toolbox.

Figure~\ref{fig:vehicles_and_drivers} shows results for numbers of
stations ranging from $10$ up to $200$.  For each number of stations
we generate $20$ random problem instances of the form described above.
The thick line in each plot shows the mean over the $20$ trials while
the thin dashed lines show the maximum and minimum values.  The left
figure shows how $V_{\alpha^*}$ and $R_{\alpha^*,\beta^*}$ vary with
the number of stations. The middle figure shows the ratio
$R_{\alpha^*,\beta^*}/V_{\alpha^*}$ as a function of the number of
stations.  We can see that we need between $1/3$ and $1/4$ as many
drivers as we do vehicles.  The right figure shows the ratio between
the minimum number of rebalancing vehicles in transit and the number
of drivers.  This gives a measure of the fraction of drivers that are
driving rebalancing vehicles (versus rebalancing themselves).  It is
interesting to note that this ratio is quite low, reaching
approximation $1/5$ for $200$ stations.

\begin{figure*}
\centering
\includegraphics[width=0.45\linewidth]{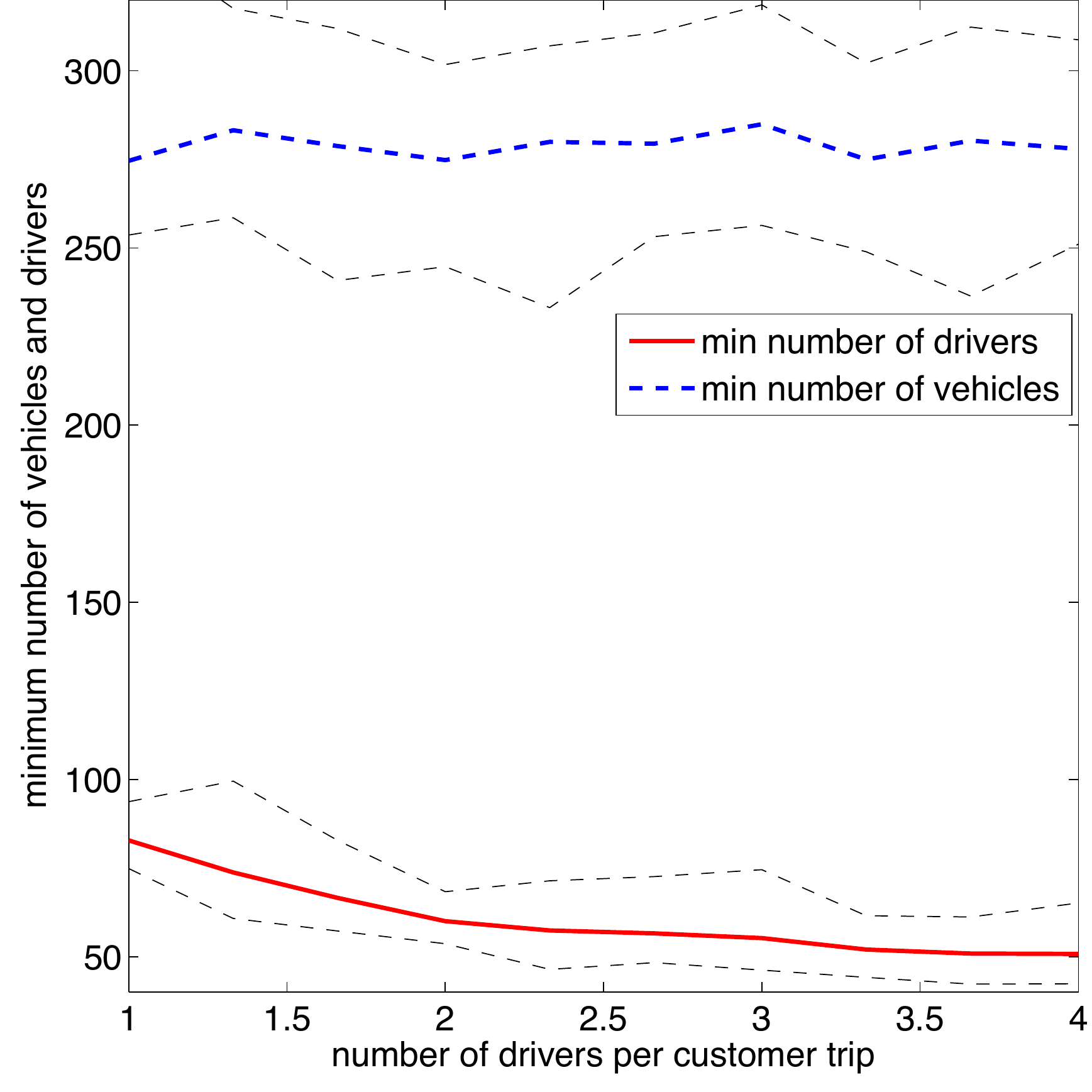} \hfill
\includegraphics[width=0.45\linewidth]{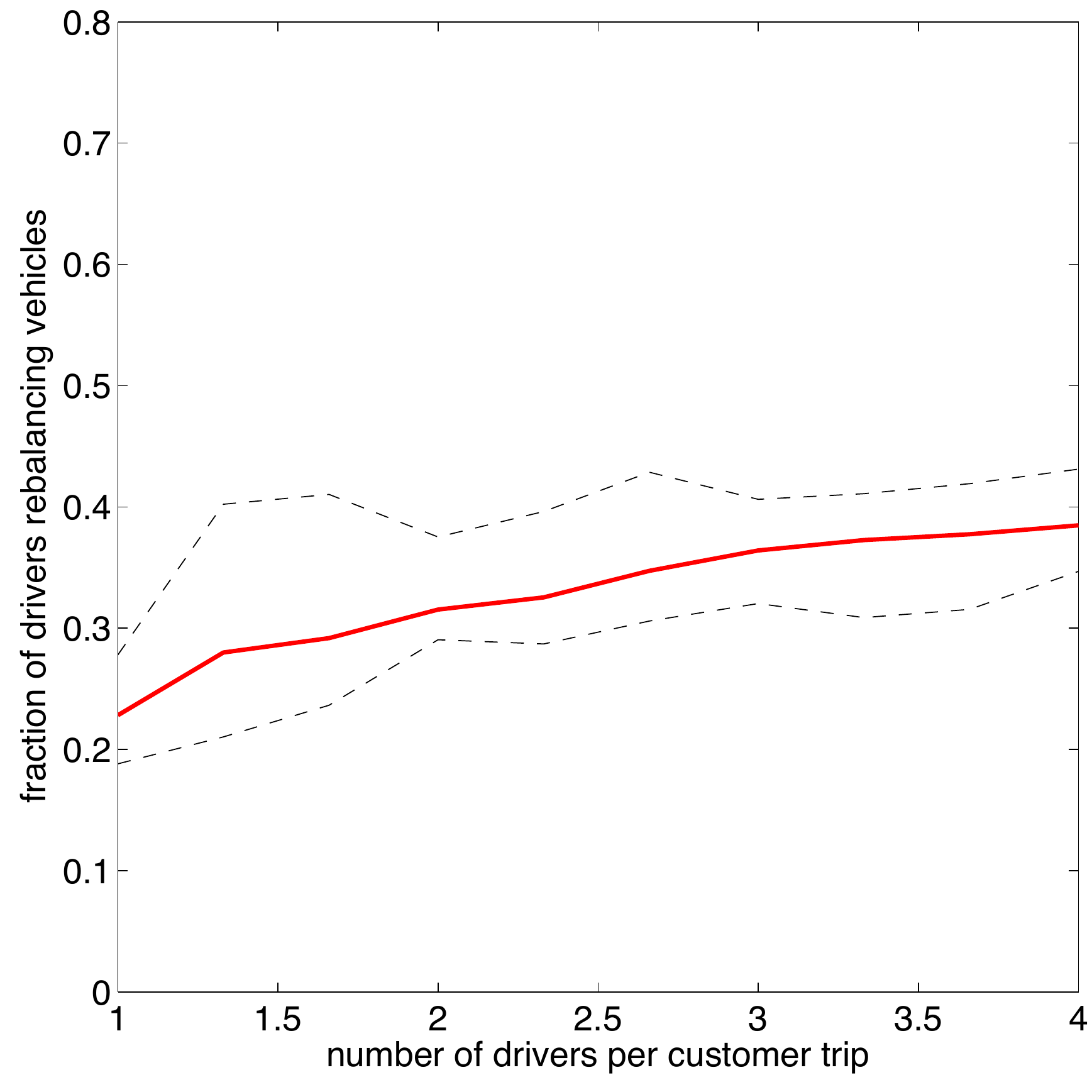} 
\caption{Increasing the number of drivers per customer trip for $100$
  station problems.  Left figure: The minimum number of vehicles and
  drivers.  Right figure: The fraction of drivers that are performing
  vehicle rebalancing trips.}
\label{fig:fij_exp}
\end{figure*}
One way to increase the fraction of drivers performing vehicle
rebalancing is to allow multiple drivers to take a trip with a
customer.  This allows drivers to take more efficient routes back to
stations that are in need of drivers.  In our model it corresponds to
setting $f_{ij} >1$.  This is explored in Figure~\ref{fig:fij_exp}
where we range $f_{ij}$ from $1$ to $4$ for $20$ problem instances
on $100$ stations.  We can see that as we increase $f_{ij}$ from $1$
to $4$, the number of drivers decreases from approximately $80$ to
$50$, and the fraction of drivers performing vehicle rebalancing increases
from under $1/4$ to nearly~$2/5$.

\section{Conclusions}
\label{sec:conc}

In this paper we studied the problem of rebalancing the rebalancers in
a mobility-on-demand system, which blends customer-driven vehicles
with a taxi service. For a fluid model of the system, we showed that
the optimal rebalancing policy can be found as the solution of two
linear programs.  Also, we showed that in Euclidean network topologies
one would need between 1/3 and 1/4 as many drivers as vehicles, and
that this fraction decreases to about 1/5 if one allows up to 3-4
drivers to take a trip with a customer. These results could have an
immediate impact on existing one-way car-sharing systems such as
Car2Go. For future work we plan to analyze a stochastic queueing model
and study the time-varying case whereby the system's parameters change
periodically (thus modeling the day/night variations). Also, we plan
to develop real-time rebalancing policies that do not require any a
priori information, and to enrich our model by including uncertainty
in the travel times, time windows for the customers, and capacity
constraints for the roads. Finally, we are interested in using dynamic
pricing to provide incentives for customers to perform rebalancing
trips themselves.

\section*{Funding}
This research was supported by the Future Urban Mobility project of
the Singapore-MIT Alliance for Research and Technology (SMART) Center,
with funding from Singapore's National Research Foundation; and by the
Office of Naval Research [grant number N000140911051].


\newcommand{\noopsort}[1]{} \newcommand{\printfirst}[2]{#1}
  \newcommand{\singleletter}[1]{#1} \newcommand{\switchargs}[2]{#2#1}

\end{document}